\documentclass[reqno,11pt]{amsart}

\usepackage{amsfonts, amssymb, verbatim}
\usepackage{amssymb,epsfig,amscd, verbatim}
\usepackage{amsthm}
\usepackage{amsmath}
\usepackage{accents}
\usepackage{graphicx}
\usepackage{latexsym}

\usepackage[all]{xy}

\newtheorem{theorem}{Theorem}

\newtheorem{proposition}[theorem]{Proposition}

\theoremstyle{definition}
\newtheorem{definition}[theorem]{Definition}

\numberwithin{equation}{section} \numberwithin{theorem}{section}

\theoremstyle{remark}

\def\z{\,_{\dot z}\,}

\def\vac{|0\rangle}                            

\newcommand\lbb[1]{\label{#1}}

\def\as{associative}

\def\tt{\otimes}                               
\def\<{\langle}
\def\>{\rangle}

\def\d{\partial}

\def\p{\partial}

\def\cp{\mathbb{C}[\partial]}
\def\la{\lambda}

\def\vac{\mathbf{1}}                            

\newcommand{\FF}{\mathbb{F}}
\newcommand{\kk}{\mathbb{C}}

\def\be{\beta}

\def\De{\Delta}

\def\la{\lambda}

\def\z{\,_{\dot z}\,}

\def\oH{{\otimes}_H}

\def\oM{\overset    { \,_{\tiny{M}}\,  }{*}}

\def\o1{\overset    { \,_{\tiny{(1)}\, }}{*}}
\def\ooo{\overset    { \,_{\tiny{(2)}\, }}{*}}

\def\oN{\overset    {  N}{*}}
\def\oE{\overset    {  E}{*}}

\def\noi{\noindent}

\def\v2a{(V,\z,\vac,d)}
\begin{document}

\title{Cohomology of associative H-pseudoalgebras}

\author[Jos\'e I. Liberati]{Jos\'e I. Liberati$^*$}
\thanks {\textit{$^{*}$Ciem - CONICET, Medina Allende y
Haya de la Torre, Ciudad Universitaria, (5000) C\'ordoba -
Argentina. \hfill \break \indent e-mail: joseliberati@gmail.com
\hfill \break \indent Keywords: associative pseudoalgebra, associative conformal algebra, cohomology.
\hfill \break \indent ORCiD: 0000-0002-5422-4056
\hfill \break \indent Final version accepted for publication.}}
\address{{\textit{Ciem - CONICET, Medina Allende y
Haya de la Torre, Ciudad Universitaria, (5000) C\'ordoba -
Argentina. \hfill \break \indent e-mail: joseliberati@gmail.com}}}



\maketitle

\begin{abstract}
We define cohomology of associative $H$-pseudoalgebras, and we show that it describes module extensions, abelian pseudoalgebra extensions, and pseudoalgebra first order deformations. We describe in details the same results for the special case of associative conformal algebras.
\end{abstract}


\section{Introduction}\lbb{intro}

Since the pioneering papers \cite{BPZ} and \cite{Bo}, there has been a great
deal of work towards understanding of the algebraic structure underlying the
notion of the operator product expansion (OPE) of chiral fields of a
conformal field theory.  The singular part of the OPE encodes the commutation
relations of fields, which leads to the notion of a Lie conformal algebra introduced by V. Kac
\cite{K}.
In the past few years a structure theory \cite{DK}, representation theory
\cite{CK, CKW} and cohomology theory \cite{BKV} of finite Lie conformal algebras
has been developed.

In \cite{BDK}, Bakalov, D’Andrea and Kac develop a theory of “multi-dimensional” Lie conformal algebras,
called Lie $H$-pseudoalgebras, where $H$ is a Hopf algebra. They also solve classification problems and develop the cohomology theory. In \cite{BDK2,BDK3}, they continue with the representation theory, classifying the  irreducible modules over finite simple Lie $H$-pseudoalgebras.

In the present work, we study {\it associative}  $H$-pseudoalgebras  and the particular case of associative conformal algebras, that is, when $H=\cp$. The associative $H$-pseudoalgebras has not been studied to the extent it needs. Important results for associative conformal algebras has been obtained by P. Kolesnikov (see \cite{Ko}), where an  analog of the Wedderburn
theorem for associative conformal algebras was proved. In \cite{D1}, I. A. Dolguntseva   define the cohomology groups of associative $H$-pseudoalgebras, and prove an analog of Hochschild's theorem for such algebras, establishing a relationship between extensions of the algebras and the second cohomology group. The explicit computations of the second cohomology group for the main examples of associative conformal algebras,  $Cend_n$ and $Cur_n$ are present in \cite{D2}. In \cite{Ko2}, the classification of irreducible subalgebras of the associative conformal algebra $Cend_n$ is presented.
 In \cite{KK,Koz}, they  describe all semisimple algebras of conformal
endomorphisms which have the trivial second Hochschild cohomology group with coefficients in every conformal bimodule. As a consequence, they state a complete solution of the radical splitting problem in the class of associative conformal algebras with a finite faithful representation.
  In \cite{BKL}, we describe the finite irreducible modules over $Cend_{n,p}$ (a family of infinite subalgebras of $Cend_n$). We also classify certain extensions of irreducible modules over $Cend_{n,p}$. We also obtained all the automorphism of $Cend_{n,p}$.

As we pointed out, the cohomology of associative $H$-pseudoalgebras was defined in \cite{D1}, but they use it only to describe the extensions of algebras using the second cohomology group.
In the present  work, we develop in full details   the zero, first and second cohomologies of associative $H$-pseudoalgebras.

The zero cohomology deserve special attention. The zero differential map $d_0$ is not explicitly written in any paper, and the general formula for the differentials maps given in \cite{D1} does not apply. So, this is the first time where the zero cohomology group is described. The image of $d_0$ is what we call the set of inner derivations, and we prove that they are  derivations, that is, we present a proof that the composition of differentials $d_1\circ d_0$ is zero. This is one of the new results of this work.

For an associative $H$-pseudoalgebra  $A$, and for any pair of left $A$-modules $M$ and $N$, we provide a new structure of $A$-bimodule on Chom$(M,N)$, where Chom$(M,N)$ is the conformal analog of the Hom functor for associative algebras (see \cite{BDK}). Then, one of our main results is Theorem \ref{th1}, where we obtained that the extensions of modules, of $M$ by $N$, is in one-to-one correspondence with elements of the first cohomology group of $A$ with coefficient in Chom$(M,N)$.

Finally, we present another main result, given by the classification of first order deformations of an associative $H$-pseudoalgebra in terms of the second cohomology group, see Theorem \ref{th3}.

At the end of this work we apply these results to the particular example of associative conformal algebras. In this case, the $n$-cochains are defined using only $n-1$ variables, instead of the  $n$-variables used in the Lie conformal algebra case in \cite{BKV}. Our situation is similar to the corrected version presented in \cite{DeK}.

In section 2, we   present the basic definitions and notations. In section 3, we define  the Hochschild cohomology for an associative $H$-pseudoalgebra $A$ over an $A$-bimodule. Then, we study  in more details the zero, first and second cohomologies.
 In section 4, we describe the extensions of modules over an associative $H$-pseudoalgebra.
In  section 5, we describe the abelian extensions and the first order deformations in terms of the corresponding second cohomology group.
In section 6, we apply these results to the particular example of associative conformal algebras.

Unless otherwise specified, all vector spaces, linear maps and
tensor products are considered over a field $\mathbb{F}$ of characteristic 0.

\section{Definitions and notation}\lbb{def }

\

Let $H$ be a Hopf algebra with comultiplication $\De$ and  counit $\varepsilon$. A more conceptual approach to the theory of associative conformal algebras, their identities, modules, cohomology, etc., is provided by the notion of an $H$-pseudoalgebra introduced in \cite{BDK}. Indeed, in ordinary algebra, all basic definitions may be stated in terms of linear spaces, polylinear maps, and their compositions. For $H$-pseudoalgebras, the base field is replaced with the Hopf algebra $H$, the class of linear spaces is replaced with the class  $\mathcal{M}(H)$ of left $H$-modules and the role of $n$-linear maps is played by $H^{\tt n}$-linear maps of the form
\begin{equation*}
  \varphi:V_1\tt\dots \tt V_n\longrightarrow H^{\tt n}\oH V, \qquad V_i,V\in \mathcal{M}(H),
\end{equation*}
where $H^{\tt n}=H\tt\dots \tt H$ and we define the right action of $H$ on $H^{\tt n}$ by setting
\begin{equation*}
  (h_1\tt \dots \tt h_n)\cdot h=(h_1\tt \dots \tt h_n)\, \De^{(n-1)}(h),
\end{equation*}
where
\begin{equation*}
  \De^{(n-1)}:=(\De \tt {\rm id} \tt \dots \tt {\rm id})\dots (\De \tt {\rm id})\De: H\longrightarrow H^{\tt n}
\end{equation*}

\vskip .2cm

\noi is the iterated comultiplication for $n> 1$, and $\De^{(0)}:={\rm id}$. The map $\varphi$ is called $H^{\tt n}$-linear if
\begin{equation*}
  \varphi(h_1 a_1\tt \dots \tt h_n\, a_n)=\big((h_1\tt \dots \tt h_n)\oH 1\big) \, \varphi(a_1\tt\dots \tt a_n)
\end{equation*}
for $h_i\in H$ and $a_i\in V_i$.

Let $V_1,V_2$ and $V_3$ be left $H$-modules on which some $H^{\tt 2}$-linear operation $*:V_1\tt V_2\longrightarrow H^{\tt 2}\oH V_3$ is defined. Note that $*$ naturally extends to
\begin{equation*}
  *:\big(H^{\tt n} \oH V_1\big)\tt \big(H^{\tt m} \oH  V_2\big) \longrightarrow H^{\tt {(n+m)}}\oH V_3
\end{equation*}
by taking
\begin{align}\label{1}
 & \big((h_1\tt \dots \tt h_n)\oH\,  v_1\big)*\big((g_1\tt \dots \tt g_m)\oH \, v_2\big)
  =  \\
   & =\big((h_1\tt \dots \tt h_n \tt g_1\tt \dots \tt g_m)\oH \, 1\big)\big(   (\De^{(n-1)}\tt \De^{(m-1)}   )\oH \, {\rm id}\big)(v_1* v_2).\nonumber
\end{align}

\vskip .2cm

\noi This formula reflect the composition rule of polylinear maps in $\mathcal{M}(H)$ (see \cite{BDK} for details).

\

An $H$-{\it pseudoalgebra} is a left $H$-module $A$ together with an $H^{\tt 2}$-linear map
\begin{align*}
   & *:A\tt A\longrightarrow H^{\tt 2}\oH A \\
   & \ \  \ \ \  a\tt b\ \longmapsto \ a*b
\end{align*}
called the {\it pseudoproduct} (similar to the definition of an ordinary algebra as a linear space equipped with a bilinear product map).

In order to define associativity of a pseudoproduct, we  extend it from $A\tt A\longrightarrow H^{\tt 2}\oH A$ to $(H^{\tt 2}\oH A)\tt A\longrightarrow H^{\tt 3}\oH A$, and to $A\tt (H^{\tt 2}\oH A)\longrightarrow H^{\tt 3}\oH A$, by using the composition rules in (\ref{1}) with $A=V_1=V_2=V_3$:

\begin{align*}
  (f\oH\, a)*b & = \sum_i \, (f\tt 1)\, (\De\tt {\rm id})\, (g_i)\,\oH \, c_i \\
  a*(f\oH\, b) & = \sum_i \, (1\tt f)\, ({\rm id}\tt \De)\, (g_i)\,\oH \, c_i
\end{align*}
where $a*b=\sum_i\, g_i\, \oH \, c_i$.

An $H$-pseudoalgebra is called {\it associative} if it satisfies the usual equality (in $H^{\tt 3}\oH A$):

\begin{equation}\label{2}
  (a*b)*c=a*(b*c).
\end{equation}

\vskip .2cm

\noi In more details, each term of (\ref{2}) is explicitly given by the following formulas: if
\begin{equation*}
  a*b=\sum_i\ (f_i\tt g_i)\, \oH \, e_i, \quad {\rm and}\quad e_i*c=\sum_j\ (f_{ij}\tt g_{ij})\,\oH \, e_{ij},
\end{equation*}
then
\begin{equation*}
  (a*b)*c=\sum_{i,j}\, \Big(f_i f_{ij(1)}\tt\,  g_i  f_{ij(2)}\tt \, g_{ij}\big)\, \oH \,e_{ij} \, \in H^{\tt 3}\oH\, A.
\end{equation*}

\noi Similarly, if we write

\begin{equation*}
  b*c=\sum_i \ (h_i\tt l_i)\, \oH \, d_i, \quad {\rm and}\quad a*d_i=\sum_j\ (h_{ij}\tt l_{ij})\,\oH \, d_{ij},
\end{equation*}
then
\begin{equation*}
  a*(b*c)=\sum_{i,j}\, \Big(h_{ij}\tt\,  h_i \ l_{ij(1)}\tt \, l_i \ l_{ij(2)}\big)\, \oH \,d_{ij} \, \in H^{\tt 3}\oH\, A.
\end{equation*}

\

\begin{definition}
  Let $A$ be an associative $H$-pseudoalgebra.

\vskip .1cm

\noi (a) A {\it left $A$-module} is a left $H$-module $M$ together  with an $H^{\tt 2}$-linear map   $\oM:A\tt M\longrightarrow H^{\tt 2}\oH \, M$ such that

\begin{equation*}
  (a*b)\oM u=a\oM(b\oM u)
\end{equation*}
for all $a,b\in A$, and $u\in M$.

\vskip .1cm

\noi (b) A {\it right $A$-module} is a left $H$-module $M$ together  with an $H^{\tt 2}$-linear map   $\oM:M\tt A\longrightarrow H^{\tt 2}\oH\,  M$ such that

\begin{equation*}
  (u\oM a)\oM b=u \oM(a* b)
\end{equation*}
for all $a,b\in A$, and $u\in M$. In general, we shall simply write $*$ instead of $\oM$.

\vskip .1cm

\noi (c) A {\it bimodule} over $A$ is a left and right $A$-module $M$ satisfying

\begin{equation*}
  (a*u)* b=a *(u* b).
\end{equation*}
\end{definition}

\vskip .2cm

If $H=\mathbb{C}$, then all these definitions correspond to the usual \as \ algebras and their modules.

\

\section{Hochschild cohomology for  associative $H$-pseudoalgebras}\lbb{def }

\

Let us describe the Hochschild cohomology for an associative $H$-pseudoalgebra $A$ and a bimodule $M$ over $A$ (see \cite{D1}). The space of $n$-cochains $C^{\, n}(A,M)$ consists of all $H^{\tt n}$-linear maps
\begin{equation}\label{hhh}
  \varphi \, :\, A^{\tt n}\longrightarrow H^{\tt n}\, \oH\, M.
\end{equation}

\vskip .2cm

\noi The differential $d_n: C^{\, n}(A,M)\longrightarrow C^{\, n+1}(A,M)$ is defined similarly to the ordinary one, assuming the compositions of polylinear maps in $\mathcal{M}(H)$:

\begin{align}\label{3}
  \big(d_n\, \varphi\big)(a_1,\dots , a_{n+1})
  & =  a_1 * \varphi(a_2,\dots , a_{n+1})  \\
  & +\sum_{i=1}^{n}\, (-1)^i\, \varphi (a_1,\dots , a_i*a_{i+1} , \dots a_{n+1}) \ +\ (-1)^{n+1} \, \varphi(a_1,\dots , a_{n}) \, *\, a_{n+1}. \nonumber
\end{align}

\vskip .2cm

\noi In the first and the last summand in (\ref{3}), we use the following conventions that  correspond to the composition defined in (\ref{1}). If $a*u=\sum_i\, f_i\oH\, u_i\in H^{\tt\,2}\oH M$, for $a\in A, u\in M$, then for any $f\in \,H^{\tt n}$, we set

\begin{equation}\label{4}
  a*\big(f\, \oH \, u\big)= \sum_i\, (1\tt f) \big({\rm id}\tt \De^{(n-1)}\big) (f_i)\, \oH\, u_i \in H^{\tt\, (n+1)} \oH\, M.
\end{equation}

\noi Similarly, if $u*a=\sum_i\, g_i \, \oH\, u_i\in H^{\tt\,2}\oH M$, for $a\in A, u\in M$, then for any $g\in H^{\tt n}$, we set

\begin{equation}\label{5}
  \big(g\, \oH \, u\big)*a= \sum_i\, (g\tt 1) \big(\De^{(n-1)}\tt {\rm id} \big) (g_i)\, \oH\, u_i \in H^{\tt\, (n+1)} \oH\, M.
\end{equation}

\noi Finally, it remains to describe the composition used in the second summand in (\ref{3}). For $g\in H^{\tt\,2}$ and $\varphi\in C^{\, n}(A,M)$, we set

\begin{align}\label{7}
 & \varphi(b_1,\dots ,b_{\, i-1}, g\, \oH \, b_i, b_{\, i+1}, \dots ,b_n)=\\
 & =\Big[\, \big( 1^{\tt (i-1)}\tt\, g\, \tt\, 1^{\tt (n-i)}\big)\,
  \big({\rm id}^{\tt (i-1)}\tt\, \De\, \tt\, {\rm id}^{\tt (n-i)}\big)\, \oH\, {\rm id}_M \Big]\, \varphi(b_1,\dots ,b_n)
  \in  H^{\tt\,(n+1)}\oH M.\nonumber
\end{align}

\vskip .2cm

\noi  Direct computations show that $d_{n+1}\circ \, d_n =0$.
If $d_n\,\varphi=0$, then $\varphi$ is called an $n$-$cocycle$. A cochain $\varphi\in C^{\, n}(A,M)$ is called an $n$-$coboundary$ if there exists an $(n-1)$-cochain $\psi$ such that $d_n\,\psi=\varphi$. Denote by $Z^{\, n}(A,M)$ and $B^{\, n}(A,M)$ the subspaces of $n$-cocycles and $n$-coboundaries, respectively. The quotient space $H^{\, n}(A,M)=Z^{\, n}(A,M)/B^{\, n}(A,M)$ is called the $n$-th Hochschild cohomology group of $A$ with coefficients in $M$.

\vskip .2cm

Let us see in more details the zero, first and second cohomologies. The case $n=0$ deserve special attention. It is not explicitly written in any work.  We shall assume that $A^{\tt \, 0}=\mathbb{F}=H^{\tt\, 0}$. Then, the $0$-cochain $\varphi\in C^{\, 0}(A,M)$ is a map
\begin{equation*}
  \varphi\, :\, \mathbb{F}\longrightarrow \mathbb{F}\, \oH\, M.
\end{equation*}

\vskip .2cm

\noi Hence, $\varphi$ is fully determined by $\varphi(1)\in \mathbb{F}\,\oH\, M\simeq M/H^+ M$, where $H^+=\{h\in H\, |\, \varepsilon(h)=0\}$ is the augmentation ideal, and $\FF\cdot h:=\FF\, \varepsilon(h)$. Therefore,

\begin{equation*}
  C^{\, 0}(A,M)\simeq M/H^+ M.
\end{equation*}
\vskip .2cm

\noi Observe that $C^{\,1}(A,M)=\ $Hom$_H(A,M)$ and the differential  $d_0\, :\,  C^{\,0}(A,M)\longrightarrow C^{\, 1}(A,M)$ is defined by the following formula: if $\varphi\in C^{\, 0}(A,M)$ and $u_\varphi :=\varphi(1)\in M$, then

\begin{equation*}
  \big(d_0 \, \varphi\big)(a)=\sum_i\, ({\rm id}\tt\varepsilon)(h_i) \, u_i \, -\, \sum_j  (\varepsilon\tt {\rm id}) (l_j)\, v_j\in M,
\end{equation*}
where $a*u_\varphi=\sum_i\, h_i\oH \, u_i\in H^{\tt\,2}\oH M$ and
 $u_\varphi*a=\sum_j\, l_j\oH \, v_j\in H^{\tt\,2}\oH M$, for $a\in A$, or in a simpler form, we have
 \begin{equation}\label{6}
   \big(d_0 \, \varphi\big)(a)=\big[ ({\rm id}\tt\varepsilon)\, \oH\, {\rm id}_M\big] \, (a \oM u_\varphi ) -
   \big[ (\varepsilon\tt{\rm id})\, \oH\, {\rm id}_M\big] \, (u_\varphi  \oM a).
 \end{equation}

\vskip .2cm

\noi It is clear that $d_0$ is well defined:  If $\varphi(1)=1\oH  h u$, with $\varepsilon(h)=0$, then  we simple have to  use $a\oM hu=((1\tt h)\oH 1)(a\oM u)$ and $hu\oM a=((h\tt 1)\oH 1)(u\oM a)$ in (\ref{6}) to get the result. Similarly, it is easy to see that $d_0\, \varphi\in C^1(A,M)$.

\vskip .2cm

Therefore, we obtain

\begin{equation*}
 H^{0}(A,M)=\Big\{ u\in M/H^+ M\, | \,  \Big[ ({\rm id}\tt\varepsilon)\, \oH\, {\rm id}_M\Big] \, (a \oM u ) =
   \Big[ (\varepsilon\tt{\rm id})\, \oH\, {\rm id}_M\Big] \, (u  \oM a),\ \mathrm{for \ all} \ a\in A\Big\}.
\end{equation*}

\vskip .2cm

Now, recall that $C^{\,1}(A,M)=\ $Hom$_H(A,M)$, since we identified $H\oH M\simeq M$. Observe that

\begin{equation*}
  C^{\,2}(A,M)=\Big\{ \varphi : A\tt A\rightarrow H^{\tt 2}\oH M \, | \, \varphi(h a, g b)=\big((h\tt g)\oH 1\big) \varphi (a,b), \forall\ a,b\in A, \forall\ h,g\in H\Big\}
\end{equation*}

\vskip .2cm

\noi and the differential is given by
$
  \big(d_1  \varphi\big)(a,b)=a*\varphi(b)-\varphi(a*b)+\varphi(a)*b.
$
Using the conventions (\ref{4}) and (\ref{5}), it is clear that

\begin{equation*}
  \big(d_1  \varphi\big)(a,b)=a\oM\varphi(b)-\varphi(a*b)+\varphi(a)\oM b,
\end{equation*}

\vskip .2cm

\noi and it remains to prove that the composition (\ref{7}) means that $\varphi(a*b)=
({\rm  id}_{H^{\tt 2}}\oH \varphi)(a*b)$, that is, we have to consider  the trivial extension of $\varphi$ to a map from ${H^{\tt 2}}\oH A$ to $M$. In fact, if $a*b=\sum_j \, g_j\oH c_j$ with $g_j\in H^{\tt\, 2}$ and $c_j\in A$, then using (\ref{7}), we have
\begin{equation}\label{nuevo}
  \varphi (a*b)=\sum_j \, \varphi (g_j\oH c_j)=\sum_j\, (g_j \De \oH {\rm id}_M) \, \varphi(c_j)=\sum_j\, g_j\oH \varphi(c_j)=
  ({\rm  id}_{H^{\tt 2}}\oH \varphi)(a*b),
\end{equation}
and in the middle of (\ref{nuevo}) we have used that $\varphi(c_j)\in M$ since we identified $H\oH M$ with  $M$.

A map $f\in \,$Hom$_H(A,M)$ is called a {\it derivation} from $A$ to $M$ if
\begin{equation*}
  f(a*b)=a\oM f(b)+f(a)\oM b,
\end{equation*}

\vskip .1cm

\noi for all $a,b\in A$ and $f$ extended trivially to a map from ${H^{\tt 2}}\oH A$ to $M$. We denote by Der$(A,M)$ the set of all derivations from $A$ to $M$. Then Ker$\, d_1=\,$Der$(A,M)$.

\vskip .2cm

\begin{proposition}\label{p1}
  For $u\in M$, we define $f_u:A\rightarrow M$ by
  \begin{equation*}
    f_u(a)=\big[ ({\rm id}\tt\varepsilon)\, \oH\, {\rm id}_M\big] \, (a * u ) -
   \big[ (\varepsilon\tt{\rm id})\, \oH\, {\rm id}_M\big] \, (u  * a).
  \end{equation*}
  Then $f_u$ is $H$-linear and it is a derivation. Hence, we have that $d_1\circ \, d_{\,0} =0$.
\end{proposition}

\begin{proof}
   First, we prove that it is $H$-linear:
   \begin{align*}
     f_u(ha) = & \big[ ({\rm id}\tt\varepsilon)\, \oH\, {\rm id}_M\big] \, \big((h\tt 1)\oH 1\big) (a * u ) -
   \big[ (\varepsilon\tt{\rm id})\, \oH\, {\rm id}_M\big] \, \big((1\tt h)\oH 1\big)(u  * a).  \\
     = &  \big( h\oH {\rm id}_M\big) \Big(\big[ ({\rm id}\tt\varepsilon)\, \oH\, {\rm id}_M\big] \, (a * u ) -
   \big[ (\varepsilon\tt{\rm id})\, \oH\, {\rm id}_M\big] \, (u  * a)\Big) = h\, f_u(a).
   \end{align*}

\noi In order to prove that it is a derivation, observe that
\begin{align}\label{t1}
  a*f_u(b) & = a*\Big(\big[ ({\rm id}\tt\varepsilon)\, \oH\, {\rm id}_M\big] \, (b * u )\Big) -
   a*\Big( \big[ (\varepsilon\tt{\rm id})\, \oH\, {\rm id}_M\big] \, (u  * b)\Big)\\
  f_u(a)*b & = \Big( \big[ ({\rm id}\tt\varepsilon)\, \oH\, {\rm id}_M\big] \, (a * u )\Big)*b -
   \Big(\big[ (\varepsilon\tt{\rm id})\, \oH\, {\rm id}_M\big] \, (u  * a)\Big)*b\label{t2}\\
  f_u(a*b) & =  \sum_i \, (f_i\tt g_i)\oH
  \Big(\big[ ({\rm id}\tt\varepsilon)\, \oH\, {\rm id}_M\big] \, (e_i * u ) -
   \big[ (\varepsilon\tt{\rm id})\, \oH\, {\rm id}_M\big] \, (u  * e_i)\Big),
  \label{t3}
\end{align}

\noi where $a*b=\sum_i\, (f_i\tt g_i)\oH \, e_i$,  and in (\ref{t3}) we used (\ref{nuevo}).
Now, let us see that the first term of (\ref{t1}) is equal to the first term of (\ref{t3}). If $b*u=\sum_i\, (h_i\tt l_i)\,\oH \,u_i$, then
$$
\big[ ({\rm id}\tt\varepsilon)\, \oH\, {\rm id}_M\big] \, (b * u ) = \sum_i\, h_{i}\, \varepsilon(l_{i})\, u_{i}\in M.
$$
Hence, using that $a*u_i=\sum_j\, (h_{ij}\tt l_{ij})\, \oH \, u_{ij}$, we obtain that  the first term of (\ref{t1}) is equal to

\begin{align}\label{11}
  \sum_i\, a*(h_{i}\, \varepsilon(l_{i})\, u_{i})
   & =\sum_{i,j}\,  \Big(h_{ij}\tt h_i\,\varepsilon(l_i) \, l_{ij}\Big)\oH \, u_{ij}=\sum_{i,j}\, \Big(h_{ij}\tt h_i\, l_{ij(1)}\, \varepsilon\big(l_{ij(2)}\big)\, \varepsilon(l_i)\Big)\oH\, u_{ij} \nonumber \\
   & = \sum_{i,j}\, ({\rm id}\tt {\rm id}\tt \varepsilon)\Big(h_{ij}\tt h_i\, l_{ij(1)}\, \tt\, l_{ij(2)}\, l_i\Big)\oH\, u_{ij}  \nonumber\\
   & = \big[({\rm id}\tt {\rm id}\tt \varepsilon)\oH\, {\rm id}_M\big]\big(a*(b*u)\big).
\end{align}

\vskip .4cm

\noi Now, if $e_i* u=\sum_j \, (f_{ij}\tt g_{ij})\oH \, v_{ij}$, then $\big[ ({\rm id}\tt\varepsilon)\, \oH\, {\rm id}_M\big] \, (e_i * u ) = \sum_j\, f_{ij}\, \varepsilon(g_{ij})\, v_{ij}\in M$. Hence, we obtain that  the first term of (\ref{t3}) is equal to

\begin{align*}
 \sum_{i,j} \, (f_i\tt g_i)  \, \De(f_{ij}\, \varepsilon(g_{ij})) \oH\, v_{ij}
   & = \sum_{i,j} \, \Big(f_i f_{ij(1)} \! \tt\, g_i f_{ij(2)} \, \varepsilon(g_{ij}) \Big) \oH\, v_{ij} \nonumber \\
   & = \big[({\rm id}\tt {\rm id}\tt \varepsilon)\oH\, {\rm id}_M\big]\big((a*b)*u\big),
\end{align*}
which is equal to (\ref{11}), proving that the first term of (\ref{t1}) is equal to the first term of (\ref{t3}).

\vskip .1cm

Similarly, with the same ideas, one can prove that the second term of (\ref{t1}) is equal to the first term of (\ref{t2}), and the second term of (\ref{t2}) is equal to the second term of (\ref{t3}). More precisely, it is possible to prove that

\begin{align*}
  a*f_u(b) & = \big[ ({\rm id}\tt{\rm id}\tt\varepsilon)\, \oH\, {\rm id}_M\big] \, (a*(b * u )) -
    \big[ ({\rm id}\tt\varepsilon\tt{\rm id})\, \oH\, {\rm id}_M\big] \, (a* (u  * b))\\
  f_u(a)*b & =  \big[ ({\rm id}\tt\varepsilon\tt {\rm id})\, \oH\, {\rm id}_M\big] \, ((a * u )*b) -
   \big[ (\varepsilon\tt{\rm id}\tt {\rm id})\, \oH\, {\rm id}_M\big] \, ((u  * a)*b) \\
  f_u(a*b) & =  \big[ ({\rm id}\tt{\rm id}\tt\varepsilon)\, \oH\, {\rm id}_M\big] \, ((a*b )* u ) -
    \big[ (\varepsilon\tt{\rm id}\tt{\rm id})\, \oH\, {\rm id}_M\big] \, ( u *(a * b)),
\end{align*}

\vskip .2cm

\noi obtaining that $f_u$ is a derivation.
\end{proof}

The derivations in Proposition \ref{p1} are called {\it inner derivations}, and we denote by IDer$(A,M)$ the corresponding set. Therefore, we obtain

\begin{equation*}
  H^{\, 1}(A,M)=\, {\rm Der}(A,M)/ \,  {\rm IDer}(A,M).
\end{equation*}

\vskip .2cm

\noi If $\varphi\in C^{\,2}(A,M)$, the definition of $d_2$ is clear:
\begin{equation*}
  \big( d_2\, \varphi\big) (a,b,c)=a*\varphi(b,c)-\varphi(a*b,c)+\varphi(a, b*c)-\varphi(a,b)*c.
\end{equation*}
as well as the cohomology group $H^{\, 2}(A,M)$.

\

\section{$H$-pseudolinear maps and extensions of modules over \as \ $H$-pseudoalgebra}\lbb{fff}

\

In this section, we introduce the $H$-pseudoalgebra analog of the "Hom" functor, defined in Section 10 in \cite{BDK}, and then we describe the extensions of modules over \as \ $H$-pseudoalgebra.
 The contents of this section are completely new.
\

\begin{definition}
  Let $M$ and $N$ be two left $H$-modules. An $H$-{\it pseudolinear map} from $M$ to $N$ is an $\mathbb{F}$-linear map $\phi:M\rightarrow (H\tt H)\,\oH N$ such that
  \begin{equation*}
    \phi(h u)=\big((1\tt h)\,\oH \, 1\big)\, \phi(u),\qquad h\in H, u\in M.
  \end{equation*}
\end{definition}
\noi We denote the space of all such $\phi$ by Chom$(M,N)$. We define a left action of $H$ on Chom$(M,N)$ by
\begin{equation*}
   (h \phi)( u)=\big((h\tt 1)\,\oH \, 1\big)\, \phi(u).
  \end{equation*}

\vskip .3cm

Consider the map $\rho:\,$Chom$(M,N)\tt \, M\rightarrow H^{\tt \,2}\oH N$, given by $\rho(\phi\tt u):=\phi(u)$. By definition, it is $H^{\tt \,2}$-linear, so it is a polylinear map in $\mathcal{M}(H)$, see \cite{BDK} for details. We will also use the notation $\phi*u=\phi(u)$ and consider this as a pseudoproduct or action.

\vskip .2cm

\begin{proposition}
  Let $A$ be an associative $H$-pseudoalgebra, and let $M$ and $N$ be two finite left $A$-modules. Then, we have

  \vskip .1cm

  \noi (a) Chom$(M,N)$ is a left $A$-module with the following action:
  \begin{equation*}
    \big( a*\phi\big)(u):= a*(\phi *u)
  \end{equation*}
  \noi for $a\in A, \, \phi\in $Chom$(M,N)$ and $u\in M$, where the composition rules are those defined in (\ref{1}).

  \vskip .1cm

  \noi (b)  Chom$(M,N)$ is a right $A$-module with the following action:
  \begin{equation*}
    \big( \phi*a\big)(u):= \phi *(a *u).
  \end{equation*}

  \vskip .1cm

  \noi (c)  Chom$(M,N)$ is a bimodule over $A$.
\end{proposition}

\vskip .1cm

\noi The proof of this proposition follows immediately by the definitions of left and right modules over $A$, and the compositions  rules of polylinear maps.

\vskip .2cm

\begin{definition}
  Let $M$ and $N$ be two left $A$-modules. An {\it extension} $E$ of $N$ by $M$ is an $H$-split exact sequence of left $A$-modules
  \begin{equation*}
    0\longrightarrow M\longrightarrow E\longrightarrow N\longrightarrow 0.
  \end{equation*}
  Two extensions $E_1$ and $E_2$ are {\it equivalent} if there exists an isomorphism $h:E_1 \longrightarrow E_2$ of $A$-modules, such that the diagram
$$\begin{CD}
0@>>> M @>>\,   > E_{1} @>>\,   >N @>>>0\\
@. @V1_{M}VV @VhVV @VV1_{N}V\\
0@>>> M @>>\,   > E_{2} @>>\,   >N @>>>0,
\end{CD}$$
  is commutative.
\end{definition}

\vskip .1cm

The following theorem is one of the main results of this work.

\begin{theorem}\label{th1}
  Given two finite left $A$-modules $M$ and $N$, the set of equivalence classes of $H$-split extensions
  \begin{equation*}
    0\longrightarrow M\longrightarrow E\longrightarrow N\longrightarrow 0
  \end{equation*}

\vskip .3cm

\noi of $N$ by $M$ are in one-to-one correspondence with elements of $H^{\, 1}(A, {\rm Chom}(N,M))$.
\end{theorem}

\begin{proof}
  Let $ 0\longrightarrow M\overset    {  i}{\longrightarrow} E\overset    {  p}{\longrightarrow} N\longrightarrow 0 $ be an extension of $A$-modules, which is split over $H$. Choose a splitting $E=M\oplus N=\{(u,v)\, | \, u\in M, v\in N\}$ as $H$-modules. The fact that $i$ and $p$ are homomorphisms of left $A$-modules implies $(a\in A, u\in M, v\in N)$
  \begin{equation}\label{20}
 a\oE u=a\oM u\qquad{\rm and}\qquad   a\oE v- a\oN v:= \gamma(a)(v)\in H^{\tt\, 2}\oH M.
  \end{equation}

  \vskip .1cm

\noi Using the $H^{\tt \, 2}$-linearity of the action in the module $E$, it is easy to see  that $\gamma(a)\in \, $Chom$(A,M)$ and $\gamma:A\longrightarrow \, $Chom$(N,M)$ is $H$-linear. In other words, we have that  $\gamma\in C^{\,1}(A, $Chom$(N,M)) =
$
Hom$_{H}(A,$Chom$(N,M))$.

Using associativity of $E$, we have (for $a,b \in A, u\in M, v\in N$)

\begin{equation*}
  (a*b)*(u,v)=\big( (a*b) * u + \gamma(a*b)(v)\, ,\, (a*b) * v\big),
\end{equation*}
\noi and
\begin{align*}
  a*(b*(u,v)) & = a*\big( b * u +  \gamma(b)(v) \, ,\, b * v\big)\\
              & =\big(( a * (b * u))+a *\big( \gamma(b)(v)\big) +\gamma(a)(b* v)\, ,\, a* (b* v)\big).
\end{align*}
Subtracting these two equations and using (\ref{20}), we have
\begin{equation*}
  \gamma(a*b)(v)=a*\big( \gamma(b)(v)\big) + \gamma(a)(b*v)
\end{equation*}
and using the definition of the $A$-bimodule structure in Chom$(N,M)$, we obtain that $\gamma(a*b)(v)=\big( a*\gamma(b)\big)(v) +\big( (\gamma(a))*b\big) (v)$ for all $v\in N$. Therefore,  the associativity in $E$ is equivalent to
\begin{equation*}
  \big( d_1 \gamma)(a,b)=a*\gamma(b) - \gamma(a*b) + (\gamma(a))*b=0.
\end{equation*}

If we have two isomorphic extensions $E$ and $E'$ associated to the closed elements $\gamma$ and $\gamma'$, and we choose a compatible splitting over $H$, then the isomorphism $h:E\longrightarrow E'$ is determined by an element $\be\in\, $Hom$_H(N,M)$, that is $h:M\oplus N\rightarrow M\oplus N'$, with $h(u,v)=(u+\beta(v),v)'$. Using that
\begin{equation*}
  h(a*(u,v))=(a*u+\gamma(a)(v)+\beta(a*v)\, , \, a*v),
\end{equation*}
and
\begin{align*}
  a*(h(u,v))  & = a*(u +\beta(v)\, ,\,  v)' \\
   & = (a*u+a*(\beta(v))+\gamma'(a)(v)\, ,\,  a*v),
\end{align*}
we have
\begin{equation}\label{21}
  \gamma(a)(v)=\gamma'(a)(v)+a*(\beta(v))-\beta(a*v).
\end{equation}

\vskip .2cm

\noi Now, using that
\begin{equation}\label{23}
{\rm Hom}_H(N,M)\simeq \mathbb{F}\tt_H {\rm Chom}(N,M)\simeq C^{\, 0}(A, {\rm Chom}(N,M)),
\end{equation}
(see Remark 10.1 in \cite{BDK} for details), we need to prove that (\ref{21}) is equivalent to $\gamma=\gamma\,' + (d_0 \beta)$. In order to simplify the notation, recall that any element in $H^{\tt 2}\oH W$ can be written uniquely in the form $\sum_i\, (h_i\tt 1)\oH \,w_i$, where $\{ h_i\}$ is a fixed $\mathbb{F}$-basis of $H$.  In more details, given $\phi\in\,$Chom$(N,M)$, we define the map $\phi_1:N\rightarrow M$ as follows: if $\phi(v)=\sum_i\, (h_i\tt 1)\oH u_i$, then $\phi_1(v)=\sum_i \, \varepsilon(h_i)\, u_i$. The map $\phi_1$ is $H$-linear and establishes the isomorphism in (\ref{23}). Let $\phi\in\,$Chom$(N,M)$ such that $\phi_1=\beta$. Observe that
\begin{equation*}
  \big(d_0\, \phi\big)(a)=\big[({\rm id}\tt \varepsilon)\oH {\rm id}_{\rm Chom}\big] \big(a*\phi\big) -
  \big[(\varepsilon\tt {\rm id})\oH {\rm id}_{\rm Chom}\big] \big(\phi*a\big).
\end{equation*}
Hence, we need to prove that (for $v\in N$)
\begin{equation}\label{24}
  \Big(\big[({\rm id}\tt \varepsilon)\oH {\rm id}_{\rm Chom}\big] \big(a*\phi\big)\Big)(v)  = a*\big( \beta(v)\big)
\end{equation}
and
\begin{equation}
  \Big(\big[(\varepsilon\tt {\rm id})\oH {\rm id}_{\rm Chom}\big] \big(\phi*a\big)\Big)(v)  = \beta (a*v).\label{25}
\end{equation}

\vskip .2cm

\noi Now, we shall prove (\ref{24}), and the proof of (\ref{25}) is similar.
First of all, we need to see  that
\begin{equation}\label{26}
  \Big(\big[({\rm id}\tt \varepsilon)\oH {\rm id}_{\rm Chom}\big] \big(a*\phi\big)\Big)(v)  =
  \Big[({\rm id}\tt \varepsilon \tt {\rm id})\oH {\rm id}_{M}\Big] \Big(\big(a*\phi\big)(v)\Big),
\end{equation}
Observe that $\big[({\rm id}\tt \varepsilon)\oH {\rm id}_{\rm Chom}\big] \big(a*\phi\big)= \sum_i \, \varepsilon(g_i)\, f_i \varphi_i$, if $a*\phi=\sum_i\, (f_i\tt g_i)\oH \varphi_i$. Hence, we have that
\begin{align}\label{27}
  \Big(\big[({\rm id}\tt \varepsilon)\oH {\rm id}_{\rm Chom}\big] \big(a*\phi\big)\Big)(v) & = \sum_i \, \varepsilon(g_i)\, (f_i \varphi_i)(v)= \sum_i \,\varepsilon(g_i)\, \big[(f_i \tt 1)\oH 1_M\big](\varphi_i(v) ) \nonumber \\
   & = \sum_{i,j} \,\varepsilon(g_i)\, \big(f_i f_{ij} \tt g_{ij}\big )\oH\, u_{ij},
\end{align}
where $\varphi_i(v)=\sum_j\, (f_{ij}\tt g_{ij})\oH \,u_{ij}$.
On the other hand, using the previous notation, we have that
\begin{equation*}
  (a*\phi)(v)=\sum_{i,j} \, \big(f_i f_{ij(1)}\!\tt g_i f_{ij(2)}\!\tt g_{ij}\big)\oH \, u_{ij},
\end{equation*}
obtaining that
\begin{equation}\label{28}
  \Big[({\rm id}\tt \varepsilon \tt {\rm id})\oH {\rm id}_{M}\Big] \Big(\big(a*\phi\big)(v)\Big)=\sum_{i,j} \,\varepsilon(g_i)\, \big(f_i f_{ij} \tt g_{ij}\big )\oH\, u_{ij}.
\end{equation}
Therefore, combining (\ref{27}) and (\ref{28}),  we have proved (\ref{26}).

If $\phi(v)=\sum_i\, (h_i\tt 1)\oH u_i$ and $a*u_i=\sum_j \, (h_{j}\tt 1)\oH \, u_{ij}$, then
\begin{equation*}
  a*(\phi(v))=\sum_{i,j}\, \big(h_{j}\tt h_i\tt 1\big)\oH \, u_{ij},
\end{equation*}
and using that by definition $\big(a*\phi\big)(v)=a*(\phi(v))$, then we have
\begin{equation*}
  \Big[({\rm id}\tt \varepsilon \tt {\rm id})\oH {\rm id}_{M}\Big] \Big(\big(a*\phi\big)(v)\Big)=\sum_{i,j}\, \varepsilon(h_i) \big(h_{j}\tt 1\big)\oH u_{ij}.
\end{equation*}
On the other hand, since $\beta(v)=\phi_1(v)=\sum_i\, \varepsilon(h_i)\, u_i$, then
\begin{equation*}
  a*(\beta(v))=\sum_i\, \varepsilon(h_i)(a*u_i)=\sum_i\,  \varepsilon(h_i) \big(h_{j}\tt 1\big)\oH u_{ij},
\end{equation*}
finishing the proof of (\ref{24}).
\

Conversely, given an element of $H^{\, 1}(A,$Chom$(A,M))$, we can choose a representative $\gamma\in C^{\, 1}(A,$Chom$(A,M))$ and define an  action of $A$ on $E=M\oplus N$ by (\ref{20}), which will depend only on the cohomology class of $\gamma$, finishing the proof.
\end{proof}

\

\section{Second cohomology, abelian extensions and first order deformations}\lbb{tttt}

\

In the first part of this section we describe the abelian extensions, see \cite{D1} for details.

\begin{definition}
  An {\it abelian extension} of an associative $H$-pseudoalgebra $A$ by an $A$-bimodule $M$, is an associative $H$-pseudoalgebra $E$ in a short exact sequence
  \begin{equation*}
    0\longrightarrow M\longrightarrow E\longrightarrow A\longrightarrow 0,
  \end{equation*}
where $M*M=0$ in $E$. Two abelian extensions $E_1$ and $E_2$ are {\it equivalent} if there exists an isomorphism $f:E_1\rightarrow E_2$ such that the diagram
$$\begin{CD}
0@>>> M @>>\,   > E_{1} @>>\,   >A @>>>0\\
@. @V1_{M}VV @VfVV @VV1_{A}V\\
0@>>> M @>>\,   > E_{2} @>>\,   >A @>>>0,
\end{CD}$$
is commutative.
\end{definition}

\begin{theorem}\label{th2}
  {\rm (Proved in \cite{D1}).} The equivalence classes of $H$-split abelian extensions of $A$ by an $A$-bimodule $M$ correspond bijectively to $H^2(A,M)$.
\end{theorem}

\

The next part of this section is a new contribution.

\begin{definition}
 (a) Let $t$ be a formal variable and $(A,*)$ an associative $H$-pseudoalgebra. A {\it first order deformation} of $A$ is a family of $H$-pseudoproducts of the form
  \begin{equation*}
    a \, \hat{*}\, b=a*b+t\, f(a,b)
  \end{equation*}
with $a,b\in A$, where $f:A\tt A\rightarrow H^{\tt\, 2}\oH A$ is an $H^{\tt\, 2}$-linear map (independent of $t$), such that $(A,\hat{*})$ is a family of associative $H$-pseudoalgebras up to the first order in $t$ (i.e. modulo $t^2$). More precisely, the $H$-pseudoproduct $\hat{*}$ is an $H^{\tt\, 2}$-linear map and it satisfies
\begin{equation}\label{B}
  (a \, \hat{*} \, b) \, \hat{*} \,  c= a \, \hat{*} \, ( b \, \hat{*} \,  c)\ \ {\rm mod}\ t^2,
\end{equation}
where $H$ acts trivially on $t$.

\vskip .1cm

\noi (b) Two first order deformations $\o1$   and $\ooo$ of $A$ are {\it equivalent}  if there exists a family $\phi_t:A\rightarrow A[t]$, of $H$-linear maps of the form $\phi_t={\rm id}_A + t \, g$, where $g:A\rightarrow A$ is an $H$-linear map such that
\begin{equation}\label{C}
  \phi_t(a\! \o1 \! b)=\phi_t(a)\!  \ooo \! \phi_t(b) \ \ \ \ {\rm mod}\ t^2,
\end{equation}
for $a,b\in A$.
\end{definition}

The following theorem is the second main result of this work.

\vskip .1cm

\begin{theorem}\label{th3}
  The equivalence classes of first order deformations of an associative $H$-pseudo- algebra $A$ (leaving the $H$-action intact) correspond bijectively to $H^2(A,A)$.
\end{theorem}

\begin{proof}
  Let $(A,*)$ be an associative $H$-pseudoalgebra and let $\hat{*}$  be given by
  \begin{equation}\label{A}
    a \, \hat{*}\, b=a*b+t\, f(a,b)
  \end{equation}
with $a,b\in A$, where $f:A\tt A\rightarrow H^{\tt\, 2}\oH A$ is an $H^{\tt\, 2}$-linear map. Then, using (\ref{A}), we take the expansions in (\ref{B}) mod $t^2$. By a direct computation, we can see that the coefficient of $t^0$ corresponds exactly to the associativity property of $*$, and the coefficient of $t^1$ corresponds exactly to
\begin{equation*}
  f(a*b,c)+f(a,b)*c=f(a,b*c)+a*f(b,c).
\end{equation*}
\vskip .1cm

\noi Therefore, we have seen that (\ref{A}) is a first order deformation of $A$ if and only if $f\in Z^2(A,A)$.

Now, consider two first order deformations of $A$ given by $a \!  \o1 \!  b=a*b+t\, f_1(a,b)$ and $a \!  \ooo \!  b=a*b+t\, f_2(a,b)$. They are equivalent if and only if there exists $g\in\, $Hom$_H(A,A)$ such that $\phi_t:={\rm id}_A + t \, g$ satisfies (\ref{C}).   A direct computation shows that (\ref{C}) is equivalent to
\begin{equation*}
  f_1(a,b)-f_2(a,b)=a*g(b)-g(a*b)+ g(a)*b,
\end{equation*}
for all $a,b \in A$, Therefore, it is equivalent to $f_1-f_2=d_1 g$, finishing the proof.
\end{proof}

\

\section{Cohomology of associative conformal algebras}\lbb{tttt111}

\

In this final section, we restrict the definitions and results of the previous sections to associative conformal algebras.
Conformal algebras are exactly $H$-pseudoalgebras over the polynomial Hopf algebra $H=\cp$, with coproduct $(\Delta f)(\d)=f(\d\tt 1 + 1\tt \d )$, counit $\varepsilon (f)=f(0)$, and antipode $(S f)(\d)=f(-\d)$. The structure of a conformal algebra on  a $\cp$-module $A$ is given by a $\mathbb{C}$-linear    map  $A\tt A\rightarrow A[\la]$, $a\tt b\mapsto a_\la b$, called the $\la$-product.
The relation between pseudoproduct and   $\la$-product is given by
\begin{equation*}
  a*b=\big(a_\la b\big)_{|_{\la=- \d \tt 1}}
\end{equation*}

\noi The $H^{\tt 2}$-linearity on  $*$ corresponds to the {\it sesquilinearity}:
\begin{equation}\label{51}
  (\p a)_\la b=-\la (a_\la b),\qquad {\rm and}\ \  a_\la (\p
 b)=(\la+\p) (a_\la b).
\end{equation}
The conformal algebra is called {\it associative} if
\begin{equation*}
  (a_\la b)_{\la + \mu} \, c=a_\la (b_\mu c),
\end{equation*}
which is the restriction of the associative axiom of a pseudoproduct.

\begin{definition}
  Let $A$ be an associative conformal algebra.

  \vskip .1cm

  \noi (a) A {\it left conformal module} over $A$ is a $\cp$-module $M$ with a $\mathbb{C}$-linear map $A\otimes M \longrightarrow \kk[\la]\otimes M$,
 $a\otimes u  \mapsto a_\la u$, called the $\la\,$-action,
satisfying the properties $(a,\, b \in A,\ u\in M)$:

\vskip .3cm

\noindent    $\qquad \qquad \qquad\qquad \qquad (\p a)_\la u=-\la \, a_\la u   ,  \qquad  a_\la(\p u)=(\la +\p)\, (a_\la u) ,$

\vskip .3cm

\noindent   $\  \qquad\qquad\qquad\qquad\qquad a_\la ( b_\mu u)=(a_{ \la} b)_{\la+\mu} u$.

\vskip .3cm

\noindent (b) A {\it right conformal module} over $A$ is a $\cp$-module $M$ with a $\mathbb{C}$-linear map $M\otimes A \longrightarrow \kk[\la]\otimes M$,
 $u\otimes a  \mapsto u_\la a$, called the $\la\,$-action,
satisfying the corresponding sesquilinearity and
%
%
\begin{equation*}
 u_\la ( a_\mu b)=(u_{ \la} a)_{\la+\mu} b.
\end{equation*}

\vskip .2cm

\noindent (c) A {\it conformal bimodule} $M$ over $A$ is a left and right conformal module that satisfies
\begin{equation*}
  a_\la ( u_\mu b)=(a_{ \la} u)_{\la+\mu} b.
\end{equation*}
\end{definition}

\noi The notion of conformal bimodule was introduced after Definition 1.4 in \cite{BKV}. A  conformal module is called $finite$ if it is finitely generated over $\cp$.

\vskip .2cm

Now, we describe Chom$(M,N)$ in the conformal case, that is $H=\cp$. Let $M$ and $N$ be  two $\cp$-modules. A {\it conformal linear map} from $M$ to
 $N$ is a $ \mathbb{C}$-linear map $f_\la :M\to N[\la]$,  such that
 $$
 f_\la  (\p u)= (\la + \p) \, f_\la(u),
 $$

 \vskip .2cm

\noi for  $u\in M$. We denote the vector space of all such maps by
 Chom$(M,N)$. It has an structure of  a $\cp$-module given by
$$
(\p f)_{\la}(u):= -\la \,f_{\la}(u).
$$

 \vskip .2cm

\noi If $M$ and $N$ are finite left conformal $A$-modules, then   Chom$(M,N)$ is a left conformal  $A$-module with the action (for $a\in A, u\in M$)
\begin{equation*}
(a_\la f)_\mu u : =a_\la(f_{\mu-\la}u),
\end{equation*}

\vskip .1cm

\noi and it is a right conformal  $A$-module with the action (for $a\in A,  u\in M$)
\begin{equation*}
(f_\la a)_\mu u : =f_\la(a_{\mu-\la}u).
\end{equation*}

\vskip .1cm

\noi With these structures, it is a conformal bimodule over $A$.

\vskip .3cm

In \cite{BKV}, the Hochschild cohomology group was defined and the space of $n$-cochains has $n$ variables, and it was necessary to take certain quotient.

In \cite{DK}, for the case of Lie conformal algebras, the definition was improved by taking $n-1$ variables. Following this idea, we define the Hochschild cohomology for an associative conformal algebra $A$ and a bimodule $M$ over $A$. The space of $n$-cochains $C^{\, n}(A,M)$ consists of all maps
\begin{equation*}
  \varphi_{\la _1,\dots ,\la_{n-1}} :A^{\tt n}\longrightarrow M[\la_1,\dots ,\la_{n-1}],
\end{equation*}
such that (here we use that $H^{\tt n}\oH M\simeq H^{\tt (n-1)}\tt M$ and the $H^{\tt n}$-linearity in (\ref{hhh}) translate into the following sesquilinearity properties)
\begin{equation*}
  \varphi_{\la _1,\dots ,\la_{n-1}}(a_1,\dots , \d a_i,\dots , a_n)=
  - \la_i\, \varphi_{\la _1,\dots ,\la_{n-1}}(a_1,\dots , a_n),\qquad i=1,\dots , n-1,
\end{equation*}
and
\begin{equation*}
  \varphi_{\la _1,\dots ,\la_{n-1}}(a_1,\dots , \d a_n)=(\d +\la_1 + \cdots + \la_{n-1})\, \varphi_{\la _1,\dots ,\la_{n-1}}(a_1,\dots , a_n).
\end{equation*}

\

\noi The differential turns into

\begin{align*}
  \big( d_n\, \varphi\big)_{\la _1,\dots ,\la_{n}}(a_1,\dots ,  a_{n+1}) & = (a_1)_{\la_1} \varphi_{\la _2,\dots ,\la_{n}}(a_2,\dots ,  a_{n+1}) \\
   & +\, \sum_{i=1}^n\, (-1)^i\, \varphi_{\la _1,\dots ,\la_i+\la_{i+1},\dots ,\la_{n}}(a_1,\dots, (a_i)_{\la_i} (a_{i+1}),\dots ,  a_{n+1}) \\
   & +\, (-1)^{n+1}\, \varphi_{\la _1,\dots ,\la_{n-1}}(a_1,\dots , a_n)_{(\la _1+\dots +\la_{n})} a_{n+1}.
\end{align*}

\vskip .3cm

Now, we write the details of the lowest degree cohomologies. First of all, we have $C^{\, 0}(A,M)\simeq M/\d M$ and $C^{\, 1}(A,M)=\,\, $Hom$_{\cp} (A,M)$. In order to define the differential $d_0$, we need the following ideas.
Choosing a set of generators $\{u_j \}$ of the $\cp$-module $M$, we can write for $a\in A$ and $u\in M$
\begin{equation*}
  a_\la u =\sum_k\, Q_k(\la\, ,\p )\, u_k,
\end{equation*}
where $Q_k$ are some polynomials in $\la$ and $\d$. Taking
\begin{equation*}
  P_k(x,y):=Q_k (-x,x+y),
\end{equation*}
the correspondent left pseudoaction of $A$ on $M$ is given by the $\cp^{\tt 2}$-linear map $*:A\tt M\rightarrow (H\tt H)\oH M$ defined by
\begin{equation*}
  a*u=\sum_k\, P_k(\d\tt  1 , 1\tt \d)\,\oH u_k.
\end{equation*}
We consider similar formulas for the right conformal and pseudoactions. That is, if $u_\la a =\sum_i\, S_i(\la\, ,\p )\, u_i$, then $u*a=\sum_i\, R_i(\d\tt  1 , 1\tt \d)\,\oH \, u_i$, where $R_i(x,y):=S_i(-x,x+y)$. Now, we apply formula (\ref{6}). If $\varphi(1)=u$, then
\begin{equation*}
  \big( d_0\, \varphi\big)(a)=\sum_k\, P_k(\d , 0)\, u_k\, - \, \sum_i\, R_i(0, \d)\, u_i.
\end{equation*}
Then, in the conformal case, we obtain
\begin{equation*}
  \big( d_0\, \varphi\big)(a)=\sum_k\, Q_k(-\d,\d)\, u_k -\sum_i\, S_i(0,\d)\, u_i = a_{_{-\d}} u- u_{_{\,0}} a.
\end{equation*}
Therefore, $H^{\, 0}(A,M)=\big\{u\in M/\d M\, |\, \,a_{_{-\d}} u= u_{_{\,0}} a\ \  {\rm for\  all}\ a\in A\big\}$.

\vskip .2cm

A map $f\in \,$Hom$_{\cp} (A,M)$ is called a {\it derivation} from $A$ to $M$, if
\begin{equation*}
  f(a_\la b)=a_\la f(b) + f(a)_\la b
\end{equation*}
for all $a,b\in A$. Observe that
\begin{equation*}
  C^{\, 2}(A,M)=\big\{\, \varphi_\la: A\tt A\rightarrow M[\la]\, |\, \varphi_\la(\d a,b)=-\la \,\varphi_\la (a,b)\ {\rm and }\
  \varphi_\la( a,\d b)=(\la+\d)\, \varphi_\la (a,b) \big\}
\end{equation*}
and the differential $d_1: C^{\, 1}(A,M)\rightarrow C^{\, 2}(A,M)$ is given by
\begin{equation*}
  \big( d_1\, \varphi\big)_\la (a,b)= a_\la \varphi(b) - \varphi(a_\la b)+\varphi(a)_\la b.
\end{equation*}
It is clear that Ker$\ d_1=$\, Der$(A,M)$. And the maps $g_u:A\rightarrow M$ (for $u\in M$) defined by
\begin{equation*}
  g_u(a)=a_{_{-\d}} u- u_{_{\,0}} a
\end{equation*}
correspond to the inner derivations or the image of $d_0$. By definition, we have
\begin{equation*}
  \big(d_2 \,\varphi\big)_{\la,\mu}(a,b,c)= a_\la \varphi_\mu (b, c) - \varphi_{\la+\mu} (a_\la b, c)  +  \varphi_\la (a, b_\mu c)  -
  \varphi_\la (a, b)_{\la+\mu} \, c.
\end{equation*}

\

Finally, the Theorem \ref{th1}, Theorem \ref{th2} and Theorem \ref{th3} hold for associative conformal algebras.

\

\subsection*{Acknowledgements}
The  author was supported  by a grant by Conicet, Consejo Nacional
de Investigaciones Cient\'ificas y T\'ecnicas (Argentina).  Special thanks to my teacher Victor Kac.

\

\bibliographystyle{amsalpha}


\end{document}